\documentclass[11pt]{amsart}
\usepackage{amsfonts,amsthm,amsmath}
\usepackage{natbib}
\usepackage{hyperref}

\newtheorem{theorem}{Theorem}[section]
\newtheorem{lemma}{Lemma}[section]

\newcommand{\prob}{\stackrel{P}{\longrightarrow}}
\newcommand{\eid}{\stackrel{d}{=}}
\newcommand{\one}{{\bf 1}}
\newcommand{\reals}{{\mathbb R}}
\newcommand{\bbr}{\reals}
\newcommand{\vep}{\varepsilon}

\def\Var{{\rm Var}}

\numberwithin{equation}{section}

\begin{document}

\title[Hill Estimator]{Asymptotic normality of Hill Estimator for truncated data}
\author[A. Chakrabarty]{Arijit Chakrabarty}
\address{Arijit Chakrabarty, Department of Mathematics, Indian Institute of Science, Bengaluru 560012, INDIA\\
Phone number: +91-8971868223\\
Fax: +91-80-23600146}
\email{arijit@math.iisc.ernet.in}

\begin{abstract} The problem of estimating the tail index from truncated data is addressed in \cite{chakrabarty:samorodnitsky:2009}. In that paper, a sample based (and hence random) choice of $k$ is suggested, and it is shown that the choice leads to a consistent estimator of the inverse of the tail index. In this paper, the second order behavior of the Hill estimator with that choice of $k$ is studied, under some additional assumptions. In the untruncated situation, it is well known that asymptotic normality of the Hill estimator follows from the assumption of second order regular variation of the underlying distribution. Motivated by this, we show the same in the truncated case in light of the second order regular variation.
\end{abstract}

\keywords{ heavy tails, truncation, second order regular variation, Hill estimator, asymptotic normality\vspace{.5ex}}
\thanks{Research partially supported by the NSF grant ``Graduate and Postdoctoral Training in Probability and its Applications'' at
Cornell University and the Centenary Post Doctoral Fellowship at Indian Institute of Science.}
\maketitle

\section{Introduction} Distributions with a regularly varying tail are becoming increasingly important in nature. Lots of phenomena arising in fields like telecommunications, finance and insurance exhibit the presence of such distributions. Historically, one of the most important statistical issues related to distributions with regularly varying tail is estimating the tail index $\alpha$. A detailed discussion on estimators of the tail index can be found in Chapter 4 of \cite{dehaan:ferreira:2006}. One of the most popular estimators is the Hill estimator, introduced by \cite{hill:1975}. For a one-dimensional non-negative sample $X_1,\ldots,X_n$, the Hill statistic is defined as
\begin{equation}\label{hill.defn}
h(k,n):=\frac1k\sum_{i=1}^k\log\frac{X_{(i)}}{X_{(k)}}\,,
\end{equation}
where $X_{(1)}\ge\ldots\ge X_{(n)}$ are the order statistics of $X_1,\ldots,X_n$, and $1\le k\le n$ is an user determined parameter. It is well known that if $X_1,\ldots, X_n$ are a i.i.d. sample from a distribution whose tail is regularly varying with index $-\alpha$ and $k$ satisfies $1\ll k\ll n$, then $h(k,n)$ consistently estimates $\alpha^{-1}$.
In a sense made precise by \cite{mason:1982},  the consistency of Hill statistic  is equivalent to the regular variation of the tail of the underlying distribution. Various authors have studied the second order behavior of the Hill estimator; see for example  \cite{davis:resnick:1984}, \cite{csorgo:mason:1985}, \cite{hausler:teugels:1985}, \cite{goldie:smith:1987}, \cite{geluk:dehaan:resnick:starica:1997} and \cite{dehaan:resnick:1998} among others. It is well known that if the tail of the i.i.d. random variables $X_1,\ldots,X_n$ satisfies a stronger assumption than regularly varying with index $-\alpha$, known as second order regular variation, then
$$
\sqrt k\left(h(k,n)-\frac1\alpha\right)\Longrightarrow N\left(0,\frac1{\alpha^2}\right)\,.
$$

While there are real life phenomena that do exhibit the presence of heavy tails, in lot of the cases there is a physical upper bound on the possible values. For example most internet service providers put an upper bound on the size of a file that can be transferred using an internet connection provided by them. Clearly the natural model for such phenomena is a truncated heavy-tailed distribution, a distribution which fits a heavy-tailed distribution till a certain point and then decays significantly faster. This can be made precise in the following way. Suppose that $H,H_1,\ldots$ are i.i.d. random variables so that $P(H>\cdot)$ is regularly varying with index $-\alpha$, $\alpha>0$ and that $L,L_1,L_2,\ldots$ are i.i.d. random variables independent of $(H,H_1,H_2,\ldots)$. All these random variables are assumed to take values in the positive half line. We observe the sample $X_1,\ldots, X_n$ given by
\begin{equation}\label{model}
X_j:=H_j\one(H_j\le M_n)+(M_n+L_j)\one(H_j>M_n)\,,
\end{equation}
where  $M_n$, representing the truncating threshold, is a sequence of positive numbers going to infinity. Strictly speaking, the model is actually a triangular array $\{X_{nj}:1\le j\le n\}$. However, in practice we shall observe only one row of the triangular array, and hence we denote the sample by the usual notation $X_1,\ldots,X_n$.
The random variable $L$ can be thought of to have a much lighter tail, a tail decaying exponentially fast for example. However the results of this article are true under milder assumptions.

It was observed in Chakrabarty and Samorodnitsky (2009) that if the sequence $M_n$ goes to infinity slow enough so that
\begin{equation}\label{eq.ht}
\lim_{n\to\infty}nP(H>M_n)=\infty\,,
\end{equation}
then a priori choosing a $k$ so that the Hill estimator is consistent is a problem. In order to overcome that problem, the following sample based choice of $k$ was suggested in that paper:
\begin{equation}\label{khat}
\hat k_n:=\left[n\left(\frac1n\sum_{j=1}^n\one(X_j>\gamma X_{(1)})\right)^\beta\right]\,,
\end{equation}
where $\beta,\gamma\in(0,1)$ are user determined parameters. It has been shown in that article that this choice of $\hat k_n$ leads to a consistent estimator of $\alpha^{-1}$ when \eqref{eq.ht} is true, or when that limit is zero. In this paper, we investigate the second order behavior of $h(\hat k_n,n)$ under the assumption \eqref{eq.ht} and some additional assumptions. We hope to address the case when the corresponding limit is zero in future.

In Section \ref{sec:result}, it is shown that under some assumptions, the Hill estimator with $k=\hat k_n$ is asymptotically normal with mean $1/\alpha$. In Section \ref{sec:secondorder}, we connect the assumptions of Section \ref{sec:result} to the second order regular variation of the tail of $H$. In Section \ref{sec:practice}, we comment on the issues related to using the results of sections \ref{sec:result} and \ref{sec:secondorder} in practice, and suggest ways for getting around some of them.

\section{Asymptotic normality of the Hill estimator}\label{sec:result} Suppose that we have a one-dimensional non-negative sample $X_1,\ldots,X_n$ given by \eqref{model}. We shall assume the following throughout this section.

{\bf Assumption A:} There exists a sequence $(\vep_n)$ such that
\begin{eqnarray}
\lim_{n\to\infty}P(H>M_n)^{-(1-\beta)}\vep_n&=&0\,,\label{A.1}\\
\lim_{n\to\infty}nP(H>M_n)P(L>\vep_nM_n)&=&0\,,\label{A.2}\\
\mbox{and }\lim_{n\to\infty}P(H>M_n)^{-(1-\beta)}\left\{\frac{l\left(\gamma M_n(1+\vep_n)\right)}{l(\gamma M_n)}-1\right\}&=&0\,,\label{A.3}
\end{eqnarray}
where $l(x):=x^\alpha P(H>x)$.

{\bf Assumption B:} $\lim_{n\to\infty}nP(H>M_n)=\infty$.

{\bf Assumption C:} $\lim_{n\to\infty}nP(H>M_n)^{2-\beta}(\log M_n)^2=0$.

{\bf Assumption D:} For any sequence $(v_n)$ satisfying
\begin{equation}\label{D.eq1}
v_n\sim nP(H>\gamma M_n)^\beta\,,
\end{equation}
it holds that
$$
\lim_{n\to\infty}\sqrt{v_n}\left[\frac n{v_n}P\left(H>b(n/v_n)y^{-1/\alpha}\right)-y\right]=0
$$
uniformly on compact sets in $[0,\infty)$, where
\begin{equation}\label{bdefn}
b(y):=\inf\left\{x:\frac1{P(H>x)}\ge y\right\}\,.
\end{equation}

{\bf Assumption E:} For any sequence $(v_n)$ satisfying \eqref{D.eq1},
$$
\lim_{T\to\infty}\limsup_{n\to\infty}\sqrt{v_n}\int_T^\infty\left|\frac n{v_n}P\left(H>b(n/v_n)s\right)-s^{-\alpha}\right|\frac{ds}s=0\,.
$$

The main result of this section, Theorem \ref{t1}, describes the second order behavior of $h(\hat k_n,n)$, where $h(\cdot,\cdot)$ and $\hat k_n$ are as defined in \eqref{hill.defn} and \eqref{khat} respectively, under the assumptions A-E. Of course, these assumptions are hard to check in practice. However, in Section \ref{sec:secondorder}, we show that most of these can be verified if the tail of $H$ is second order regularly varying and some additional conditions are satisfied. One could thus state the hypothesis of Theorem \ref{t1} in terms of the second order regular variation. The only reason why we decided not to do that is the following. The simplest example of a distribution with a regularly varying tail is a Pareto, which is known to not satisfy the second order regular variation as defined in \cite{resnick:2007}. Hence, if Theorem \ref{t1} is stated in terms of second order regular variation, it will not entail simple examples of regularly varying distributions like Pareto, which clearly satisfy the assumptions A, D and E.

\begin{theorem}\label{t1}
Under assumptions A,B,C,D and E,
\begin{equation}\label{t1.claim}
\sqrt{\hat k_n}\left\{h(\hat k_n,n)-\frac1\alpha\right\}\Longrightarrow N\left(0,\frac1{\alpha^2}\right)\,.
\end{equation}
\end{theorem}

The following is a brief outline of how we plan to prove this. Define
\begin{eqnarray*}
U_n&:=&\sum_{j=1}^n\one(X_j>\gamma M_n)\,,\\
V_n&:=&\sum_{j=1}^n\one(X_j>\gamma X_{(1)})\,,\\
\tilde k_n&:=&\left[n^{1-\beta}U_n^\beta\right]\,.
\end{eqnarray*}
Note that
$$
\hat k_n:=\left[n^{1-\beta}V_n^\beta\right]\,.
$$
Since we are dealing with a random sum, a natural way of proceeding is conditioning on the number of summands. However, conditioning on $V_n$ or $\hat k_n$ destroys the i.i.d. nature of the sample. Hence, we condition on $U_n=u_n$, where $(u_n)$ is any sequence of integers satisfying $u_n\sim nP(H>\gamma M_n)$. Lemma \ref{l5} is a general result, which allows us to claim weak convergence of the unconditional distribution based on that of the conditional distribution.
Clearly, by conditioning on $U_n$, $h(\tilde k_n,n)$ becomes the Hill statistic with a deterministic $k$ applied to a triangular array. The second order behavior of that is studied in Lemma \ref{l3}. In view of Lemma \ref{l5}, this translates to second order behavior of (the unconditional distribution of) $h(\tilde k_n,n)$. In order to argue the claim of Theorem \ref{t1}, all we need is showing that $h(\tilde k_n,n)$ and $h(\hat k_n,n)$ are not very far apart, and that is done in Lemma \ref{l2}. For Lemma \ref{l3} and Lemma \ref{l2}, we need that the tail empirical process, after suitable centering and scaling, converge to a Brownian Motion. This has been showed in Lemma \ref{l4}.

\begin{lemma}\label{l5} Suppose that $(B_n:n\ge1)$ is a sequence of discrete random variables satisfying
$$
\frac{B_n}{b_n}\prob1\,,
$$
for some deterministic sequence $(b_n)$. Assume that $(A_n:n\ge1)$ is a family of random variables such that whenever $\hat b_n$ is any deterministic sequence satisfying $\hat b_n\sim b_n$ as $n\longrightarrow\infty$ and $P(B_n=\hat b_n)>0$,
\begin{equation}\label{l5.eq1}
P(A_n\le\cdot|B_n=\hat b_n)\Longrightarrow F(\cdot)\,,
\end{equation}
for some c.d.f. $F$. Then $A_n\Longrightarrow F$.
\end{lemma}

\begin{proof} It suffices to show that every subsequence of $(A_n)$ has a further subsequence that converges weakly to $F$. Since every sequence that converges in probability has a subsequence that converges almost surely, we can assume without loss of generality that
\begin{equation}\label{l5.eq2}
\frac{B_n}{b_n}\longrightarrow1\mbox{ a.s.}\,.
\end{equation}
Fix a continuity point $x$ of $F$ and define a function $f_n:\bbr\longrightarrow[0,1]$ by
$$
f_n(u)=\left\{\begin{array}{ll}\frac{P(A_n\le x,B_n=u)}{P(B_n=u)},&\mbox{if }P(B_n=u)>0\\0,&\mbox{otherwise.}\end{array}\right.
$$
Clearly, for all $n\ge1$,
$$
P(A_n\le x)=Ef_n(B_n)\,.
$$
By \eqref{l5.eq1} and \eqref{l5.eq2}, it follows that
$$
f_n(B_n)\longrightarrow F(x)\mbox{ a.s.}\,.
$$
By the bounded convergence theorem, it follows that
$$
\lim_{n\to\infty}Ef_n(B_n)=F(x)\,,
$$
and this completes the proof.
\end{proof}

Throughout this section,  assumptions A, B, C, D and E will be in force.

\begin{lemma}\label{l4}Suppose that $(u_n)$ is a sequence of integers satisfying
\begin{equation}\label{l4.eq1}
u_n\sim nP(H>\gamma M_n)\,,
\end{equation}
and let
\begin{eqnarray}
\label{l2.eq2}
v_n&:=&[n^{1-\beta}u_n^\beta]-u_n\,,\\
\tilde M_n&:=&\gamma M_n\,.\label{l2.eq4}
\end{eqnarray}
Let for $n\ge1$, $Y_{n,1},\ldots, Y_{n,n}$ be i.i.d. with c.d.f. $F_n$, defined as
$$
F_n(x):=P(H\le x|H\le {\tilde M_n})\,.
$$
Then,
\begin{equation}\label{l3.eq1}
\sqrt{v_n}\left(\frac1{v_n}\sum_{i=1}^{n-u_n}\delta_{Y_{n-u_n,i}/b((n-u_n)/v_n)}(y^{-1/\alpha},\infty]-y\right)\Longrightarrow W(y)
\end{equation}
in $D[0,\infty)$, where $D[0,\infty)$ is endowed with the topology of uniform convergence on compact sets and $W$ is the standard Brownian Motion on $[0,\infty)$.
\end{lemma}

\begin{proof}For simplicity sake, denote $w_n:=n-u_n$.
It is easy to see by assumptions B and C that
\begin{equation}\label{l3.hypo1}
1\ll w_nP(H>{\tilde M_n})\ll\sqrt{v_n}\ll\sqrt {w_n}\,.
\end{equation}
Let $(\Gamma_i:i\ge1)$ be the arrivals of a unit rate Poisson Process. Define
$$
\phi_n(s):=\frac{\Gamma_{w_n+1}}{v_n}\bar F_n(s^{-1/\alpha}b(w_n/v_n))\,,
$$
where $\bar G:=1-G$ for any function $G$. By the discussion on page 24 in \cite{resnick:2007}, it follows that
\begin{equation}\label{l3.eq10}
\lim_{n\to\infty}\frac {w_n}{v_n}P(H>b(w_n/v_n))=1\,.
\end{equation}
It follows by \eqref{l3.hypo1} that
$$
\lim_{n\to\infty}\frac {w_n}{v_n}P(H>{\tilde M_n})=0\,.
$$
This in conjunction with \eqref{l3.eq10} implies that
$$
b(w_n/v_n)=o({\tilde M_n})\,.
$$
It is easy to see that $v_n$ satisfies \eqref{D.eq1}.
Hence,  for $n$ large enough,
\begin{eqnarray*}
&&\frac {w_n}{v_n}\bar F_n(s^{-1/\alpha}b(w_n/v_n))-s\\
&=&\frac1{P(H\le {\tilde M_n})}\biggl[\frac {w_n}{v_n}P\left(H>s^{-1/\alpha}b(w_n/v_n)\right)-\frac {w_n}{v_n}P(H>{\tilde M_n})\\
&&\,\,\,\,-s+sP(H>{\tilde M_n})\biggr]\,,
\end{eqnarray*}
and hence in view of Assumption D and \eqref{l3.hypo1}, it follows that for $0<T<\infty$,
\begin{equation}\label{l3.eq2}
\lim_{n\to\infty}\sqrt{v_n}\sup_{0\le s\le T}\left|\frac {w_n}{v_n}\bar F_n(s^{-1/\alpha}b(w_n/v_n))-s\right|=0\,.
\end{equation}
Also note that,
\begin{eqnarray*}
&&\sup_{0\le s\le T}\left|\phi_n(s)-\frac {w_n}{v_n}\bar F_n(s^{-1/\alpha}b(w_n/v_n))\right|\\
&=&\left|\frac{\Gamma_{w_{n}+1}}{w_n}-1\right|\frac {w_n}{v_n}\bar F_n(T^{-1/\alpha}b(w_n/v_n))\\
&=&O_p(w_n^{-1/2})O(1)\\
&=&o_p(v_n^{-1/2})\,.
\end{eqnarray*}
This in conjunction with \eqref{l3.eq2} shows that
\begin{equation}\label{l3.eq3}
\sqrt{v_n}\left(\phi_n(s)-s\right)\prob0
\end{equation}
in $D[0,\infty)$. Recall that since $1\ll v_n\ll w_n$, in $D[0,\infty)$,
$$
\sqrt{v_n}\left(\frac1{v_n}\sum_{i=1}^{w_n}\one\left(\Gamma_i\le v_ns\right)-s\right)\Longrightarrow W(s)\,;
$$
see (9.7), page 294 in \cite{resnick:2007}.
Hence, it follows by the continuous mapping theorem and Slutsky's theorem that
\begin{equation}\label{l3.eq4}
\sqrt{v_n}\left(\frac1{v_n}\sum_{i=1}^{w_n}\one\left(\Gamma_i\le v_n\phi_n(s)\right)-\phi_n(s)\right)\Longrightarrow W(s)
\end{equation}
in $D[0,\infty)$. By similar arguments as those in the proof of Theorem 9.1 in \cite{resnick:2007}, it follows that
$$
\sum_{i=1}^{w_n}\delta_{Y_{w_n,i}/b(w_n/v_n)}(y^{-1/\alpha},\infty]\eid\sum_{i=1}^{w_n}\one\left(\Gamma_i\le v_n\phi_n(s)\right)\,.
$$
This along with \eqref{l3.eq3} and \eqref{l3.eq4} shows \eqref{l3.eq1}.
\end{proof}

\begin{lemma}\label{l3} Let $(u_n)$ be a sequence of integers satisfying \eqref{l4.eq1} and let $(v_n)$ and $(\tilde M_n)$ be as defined in \eqref{l2.eq2} and \eqref{l2.eq4} respectively.
Then,
$$
\sqrt{v_n}\left(\frac1{v_n}\sum_{i=1}^{v_n}\log\frac{Y_{(n-u_n,i)}}{Y_{(n-u_n,v_n)}}-\frac1\alpha\right)\Longrightarrow N\left(0,\frac1{\alpha^2}\right)\,,
$$
where $Y_{(n,1)}\ge\ldots\ge Y_{(n,n)}$ are the order statistics of $Y_{n,1},\ldots, Y_{n,n}$, and the latter is as defined in Lemma \ref{l4}.
\end{lemma}

\begin{proof}Once again, let us denote $w_n:=n-u_n$. An application of Vervaat's lemma (Proposition 3.3 in \cite{resnick:2007}) to \eqref{l3.eq1} shows that
\begin{equation}\label{l3.eq5}
\sqrt{v_n}\left[\left\{\frac{Y_{(w_n,v_n)}}{b(w_n/v_n)}\right\}^{-\alpha}-1\right]\Longrightarrow-W(1)
\end{equation}
jointly with \eqref{l3.eq1}. This in particular, shows that
$$
\left(\sqrt{v_n}\left\{\frac1{v_n}\sum_{i=1}^{w_n}\delta_{Y_{w_n,i}/b(w_n/v_n)}(x,\infty]-x^{-\alpha}\right\},\frac{Y_{(w_n,v_n)}}{b(w_n/v_n)}\right)
$$
$$
\Longrightarrow(W(x^{-\alpha}),1)\,,
$$
in $D(0,\infty]\times\bbr$, jointly with \eqref{l3.eq5}, where $D(0,\infty]$ is also endowed with the topology of uniform convergence on compact sets. Using the continuous mapping theorem, it follows that
$$
\sqrt{v_n}\left\{\frac1{v_n}\sum_{i=1}^{w_n}\delta_{Y_{w_n,i}/Y_{(w_n,v_n)}}(x,\infty]-x^{-\alpha}\frac{Y_{(w_n,v_n)}^{-\alpha}}{b(w_n/v_n)^{-\alpha}}\right\}
$$
\begin{equation}\label{l3.eq6}
\Longrightarrow W(x^{-\alpha})\,,
\end{equation}
in $D(0,\infty]$, jointly with \eqref{l3.eq5}. As in the proof of Proposition 9.1 in \cite{resnick:2007}, we shall apply the map $\psi$ from $D(0,\infty]$ to $\bbr$, defined by
$$
\psi(f):=\int_1^\infty f(s)\frac{ds}s\,,
$$
to conclude that
\begin{equation}\label{l3.eq7}
\sqrt{v_n}\left\{\frac1{v_n}\sum_{i=1}^{v_n}\log\frac{Y_{(w_n,i)}}{Y_{(w_n,v_n)}}-\frac1\alpha\frac{Y_{(w_n,v_n)}^{-\alpha}}{b(w_n/v_n)^{-\alpha}}\right\}\Longrightarrow\int_1^\infty W(x^{-\alpha})\frac{dx}x\,,
\end{equation}
jointly with \eqref{l3.eq5}. This implies that
$$
\sqrt{v_n}\left\{\frac1{v_n}\sum_{i=1}^{v_n}\log\frac{Y_{(n,i)}}{Y_{(n,v_n)}}-\frac1\alpha\right\}\Longrightarrow\int_1^\infty W(x^{-\alpha})\frac{dx}x-\frac1\alpha W(1)
$$
as desired. Thus, it suffices to show \eqref{l3.eq7}.

To that end, note that for $1<T<\infty$, the map $\psi_T$, defined by
$$
\psi_T(f):=\int_1^Tf(s)\frac{ds}s
$$
is continuous and has compact support. Also, as $T\longrightarrow\infty$,
$$
\psi_T(W(s^{-\alpha}))\Longrightarrow\psi(W(s^{-\alpha}))\,.
$$
Some calculations will show that $\psi$ applied to the left hand side of \eqref{l3.eq6} gives the left hand side of \eqref{l3.eq7}. Thus, all that needs to be done is justifying the application of $\psi$ to \eqref{l3.eq6}, and for that, it suffices to check that for all $\epsilon>0$,
\begin{eqnarray*}
&&
\lim_{T\to\infty}\limsup_{n\to\infty}P\biggl[\sqrt{v_n}\int_T^\infty\biggl|\frac1{v_n}\sum_{i=1}^{w_n}\delta_{Y_{w_n,i}/Y_{(w_n,v_n)}}(x,\infty]\\
&&\,\,\,\,\,\,\,\,\,\,\,\,\,\,-x^{-\alpha}\frac{Y_{(w_n,v_n)}^{-\alpha}}{b(w_n/v_n)^{-\alpha}}\biggr|\frac{dx}x>\epsilon\biggr]=0\,.
\end{eqnarray*}
Note that on the set $\{Y_{(w_n,v_n)}/b(w_n/v_n)>1/2\}$,
\begin{eqnarray*}
&&\int_T^\infty\biggl|\frac1{v_n}\sum_{i=1}^{w_n}\delta_{Y_{w_n,i}/Y_{(w_n,v_n)}}(x,\infty]
-x^{-\alpha}\frac{Y_{(w_n,v_n)}^{-\alpha}}{b(w_n/v_n)^{-\alpha}}\biggr|\frac{dx}x\\
&=&\int_{TY_{(w_n,v_n)}/b(w_n/v_n)}^\infty\left|\frac1{v_n}\sum_{i=1}^{w_n}\delta_{Y_{w_n,i}/b(w_n/v_n)}(u,\infty]-u^{-\alpha}\right|\frac{du}u\\
&\le&\int_{T/2}^\infty\left|\frac1{v_n}\sum_{i=1}^{w_n}\delta_{Y_{w_n,i}/b(w_n/v_n)}(u,\infty]-u^{-\alpha}\right|\frac{du}u\,.
\end{eqnarray*}
Since $P[Y_{(w_n,v_n)}/b(w_n/v_n)\le1/2]$ goes to zero, it suffices to show that
\begin{eqnarray}
&&\lim_{T\to\infty}\limsup_{n\to\infty}P\biggl[\sqrt{v_n}\int_{T/2}^\infty\biggl|\frac1{v_n}\sum_{i=1}^{w_n}\delta_{Y_{w_n,i}/b(w_n/v_n)}(u,\infty]\nonumber\\
\label{l3.eq8}
&&\,\,\,\,\,\,\,\,\,\,\,\,\,\,-u^{-\alpha}\biggr|\frac{du}u>\epsilon\biggr]=0\,.
\end{eqnarray}
Clearly,
\begin{eqnarray*}
&&\int_{T/2}^\infty\left|\frac1{v_n}\sum_{i=1}^{w_n}\delta_{Y_{w_n,i}/b(w_n/v_n)}(u,\infty]-u^{-\alpha}\right|\frac{du}u\\
&\le&\int_{T/2}^{\infty}\left|\frac1{v_n}\sum_{i=1}^{w_n}\delta_{Y_{w_n,i}/b(w_n/v_n)}(u,\infty]-\frac {w_n}{v_n}\bar F_n\left(ub(w_n/v_n)\right)\right|\frac{du}u\\
&&+\frac {w_n}{v_n}\int_{T/2}^\infty\left|\bar F_n\left(ub(w_n/v_n)\right)-P\left(H>ub(w_n/v_n)\right)\right|\frac{du}u\\
&&+\int_{T/2}^\infty\left|\frac {w_n}{v_n}P\left(H>ub(w_n/v_n)\right)-u^{-\alpha}\right|\frac{du}u
\end{eqnarray*}
\begin{eqnarray*}
&=&\int_{T/2}^\infty\left|\frac1{v_n}\sum_{i=1}^{w_n}\delta_{Y_{w_n,i}/b(w_n/v_n)}(u,\infty]-\frac {w_n}{v_n}\bar F_n\left(ub(w_n/v_n)\right)\right|\frac{du}u\\
&&+\frac {w_n}{v_n}\int_{T/2}^{{\tilde M_n}/b(w_n/v_n)}\left|\bar F_n\left(ub(w_n/v_n)\right)-P\left(H>ub(w_n/v_n)\right)\right|\frac{du}u\\
&&+\frac {w_n}{v_n}\int_{{\tilde M_n}}^\infty P(H>u)\frac{du}u\\
&&+\int_{T/2}^\infty\left|\frac {w_n}{v_n}P\left(H>ub(w_n/v_n)\right)-u^{-\alpha}\right|\frac{du}u\\
&=:&I_1+I_2+I_3+I_4\,.
\end{eqnarray*}

Since $v_n$ is defined by \eqref{l2.eq2}, \eqref{D.eq1} holds. By Assumption E, it follows that
$$
\lim_{T\to\infty}\limsup_{n\to\infty}\sqrt{v_n}I_4=0\,.
$$
Karamata's theorem (Theorem VIII.9.1, page 281 in \cite{feller:1971}) implies that
$$
I_3=O\left(\frac {w_n}{v_n}P(H>{\tilde M_n})\right)=o\left(v_n^{-1/2}\right)\,,
$$
the second equality following from \eqref{l3.hypo1}. For $I_2$, note that
\begin{eqnarray*}
&&\bar F_n\left(ub(w_n/v_n)\right)-P\left(H>ub(w_n/v_n)\right)\\
&=&-\frac{P(H>{\tilde M_n})P\left(H\le ub(w_n/v_n)\right)}{P(H\le {\tilde M_n})}\,.
\end{eqnarray*}
Also, it is easy to see from assumption C that
\begin{equation}\label{l3.hypo2}
\lim_{n\to\infty}\frac{w_nP(H>{\tilde M_n})}{\sqrt v_n}\log\left\{\frac{{\tilde M_n}}{b(w_n/v_n)}\right\}=0\,.
\end{equation}
Thus,
$$
I_2=O\left(\frac {w_n}{v_n}P(H>{\tilde M_n})\log\frac{{\tilde M_n}}{b(w_n/v_n)}\right)=o\left(v_n^{-1/2}\right)\,,
$$
the second equality following from \eqref{l3.hypo2}.

Thus, all that remains is showing
\begin{equation}\label{l3.eq9}
\lim_{T\to\infty}\limsup_{n\to\infty}P[\sqrt{v_n}I_1>\epsilon]=0\,.
\end{equation}
Notice that
$$
E\left[\frac1{v_n}\sum_{i=1}^{w_n}\delta_{Y_{w_n,i}/b(w_n/v_n)}(u,\infty]\right]=\frac {w_n}{v_n}\bar F_n\left(ub(w_n/v_n)\right)\,.
$$
Letting $C$ to be a finite positive constant independent of $n$, whose value may change from line to line,
\begin{eqnarray*}
&&P[\sqrt{v_n}I_1>\epsilon]\\
&\le&\frac{\sqrt{v_n}}{\epsilon}E(I_1)\\
&=&C{\sqrt{v_n}}\int_{T/2}^\infty E\left|\frac1{v_n}\sum_{i=1}^{w_n}\delta_{Y_{w_n,i}/b(w_n/v_n)}(u,\infty]-\frac {w_n}{v_n}\bar F_n\left(ub(w_n/v_n)\right)\right|\frac{du}u\\
&\le&C{\sqrt{v_n}}\int_{T/2}^\infty\Var\left[\frac1{v_n}\sum_{i=1}^{w_n}\delta_{Y_{w_n,i}/b(w_n/v_n)}(u,\infty]\right]^{1/2}\frac{du}u\\
&\le&C\frac{\sqrt{w_n}}{\sqrt{v_n}}\int_{T/2}^\infty\bar F_n\left(ub(w_n/v_n)\right)^{1/2}\frac{du}u\\
&\le&C\int_{T/2}^\infty\frac{\sqrt {w_n}}{\sqrt{v_n}}P\left(H>ub(w_n/v_n)\right)^{1/2}\frac{du}u\,.\\
\end{eqnarray*}
By \eqref{l3.eq10}, the integrand clearly converges to $u^{-\alpha/2}$ as $n\longrightarrow\infty$. By \eqref{bdefn}, the integrand is bounded above by
$$
\left[\frac{P(H>ub(w_n/v_n))}{P(H>b(w_n/v_n))}\right]^{1/2}\,,
$$
which by the Potter bounds (Proposition 2.6 in \cite{resnick:2007}) is bounded above by
$
2u^{-\alpha/3}
$
for $n$ large enough. An appeal to the dominated convergence theorem shows  \eqref{l3.eq9} and thus completes the proof.
\end{proof}

\begin{lemma}\label{l2} As $n\longrightarrow\infty$,
\begin{equation}\label{l2.claim}
\sqrt{\tilde k_n}\left\{h(\tilde k_n,n)-h(\hat k_n,n)\right\}\prob0\,.
\end{equation}
\end{lemma}

\begin{proof} We start with showing that
\begin{equation}\label{l2.eq3}
\sqrt{\hat k_n}\left[\frac{\hat k_n}{\tilde k_n}-1\right]\prob0\,.
\end{equation}
In the proof of Theorem 3.2  in Chakrabarty and Samorodnitsky (2009), it has been shown that under Assumption B,
\begin{eqnarray}
\frac{U_n}{nP(H>\gamma M_n)}&\prob&1\,,\label{eq:u}\\
\frac{V_n}{nP(H>\gamma M_n)}&\prob&1\,,\label{eq:v}\\
\mbox{and }\frac{\hat k_n}{nP(H>\gamma M_n)^\beta}&\prob&1\,.\label{eq:khat}
\end{eqnarray}
In view of \eqref{eq:khat}, it suffices to show that
$$
n^{1/2}P(H>M_n)^{\beta/2}\left[\frac{\hat k_n}{\tilde k_n}-1\right]\prob0\,.
$$
Note that,
$$
\frac{n^{1-\beta}V_n^\beta}{n^{1-\beta}U_n^\beta+1}\le\frac{\hat k_n}{\tilde k_n}\le\frac{n^{1-\beta}V_n^\beta+1}{n^{1-\beta}U_n^\beta}\,,
$$
$$
\frac{n^{1-\beta}V_n^\beta}{n^{1-\beta}U_n^\beta+1}\le\left(\frac{V_n}{U_n}\right)^\beta\le\frac{n^{1-\beta}V_n^\beta+1}{n^{1-\beta}U_n^\beta}\,,
$$
and
\begin{eqnarray*}
\frac{n^{1-\beta}V_n^\beta+1}{n^{1-\beta}U_n^\beta}-\frac{n^{1-\beta}V_n^\beta}{n^{1-\beta}U_n^\beta+1}&=&\frac{n^{1-\beta}V_n^\beta+n^{1-\beta}U_n^\beta+1}{n^{1-\beta}U_n^\beta(n^{1-\beta}U_n^\beta+1)}\nonumber\\
&=&O_p\left(n^{-1}P(H>M_n)^{-\beta}\right)\\
&=&o_p\left(n^{-1/2}P(H>M_n)^{-\beta/2}\right)\,,\nonumber
\end{eqnarray*}
the equality in the second line following from \eqref{eq:u} and \eqref{eq:v}, and that in the third line following from Assumption B. Thus, it suffices to show that
$$
n^{1/2}P(H>M_n)^{\beta/2}\left[\left(\frac{V_n}{U_n}\right)^\beta-1\right]\prob0\,.
$$
By the mean value theorem, it follows that as $x\longrightarrow1$,
$$
x^\beta-1=O(|x-1|)\,.
$$
Hence, in view of the fact that $V_n/U_n$ converges to $1$ in probability, it suffices to show that
$$
n^{1/2}P(H>M_n)^{\beta/2}\left(\frac{V_n}{U_n}-1\right)\prob0\,.
$$
Using \eqref{eq:u} once again, all that needs to be shown is
$$
V_n-U_n=o_p\left(n^{-1/2}P(H>M_n)^{-(1-\beta/2)}\right)\,.
$$
Note that on the set $\{M_n\le X_{(1)}\le M_n(1+\vep_n)\}$, where $\vep_n$ is chosen to satisfy Assumption A,
$$
0\le U_n-V_n\le \sum_{j=1}^n\one\left(\gamma M_n<X_j\le\gamma M_n(1+\vep_n)\right)=:T_n\,.
$$
Thus, it suffices to show that
\begin{equation}
\lim_{n\to\infty}P(X_{(1)}\le M_n(1+\vep_n))=1\label{l1.eq2}\,,
\end{equation}
\begin{equation}
\lim_{n\to\infty}P(X_{(1)}\ge M_n)=1\label{l1.eq3}\,,
\end{equation}
\begin{equation}
\mbox{and }T_n=o_p\left(n^{-1/2}P(H>M_n)^{-(1-\beta/2)}\right)\,.\label{l1.eq4}
\end{equation}
For \eqref{l1.eq2}, note that as $n\longrightarrow\infty$,
\begin{eqnarray*}
P(X_{(1)}\le M_n(1+\vep_n))&=&\left(1-P(H>M_n)P(L>\vep_nM_n)\right)^n\longrightarrow1\,,
\end{eqnarray*}
the convergence following from \eqref{A.2} in Assumption A. This shows \eqref{l1.eq2}. For \eqref{l1.eq3}, observe that
$$
P(X_{(1)}<M_n)\le\left(1-P(H>M_n)\right)^n\,.
$$
By Assumption B, the right hand side converges to zero, and hence \eqref{l1.eq3} holds. To show \eqref{l1.eq4}, note that
$$
\Var(T_n)\le E(T_n)=np_n\,,
$$
where
$$
p_n:=P(\gamma M_n<X_1\le\gamma(1+\vep_n)M_n)\,.
$$
In view of Assumption C, for \eqref{l1.eq4}, it suffices to show that
\begin{equation}\label{l1.eq5}
p_n=o(P(H>M_n)^{2-\beta})\,.
\end{equation}
For $n$ large enough so that $\gamma(1+\vep_n)<1$,
\begin{eqnarray*}
p_n&=&P(H>\gamma M_n)-\gamma^{-\alpha}M_n^{-\alpha}(1+\vep_n)^{-\alpha}l\left(\gamma M_n(1+\vep_n)\right)\\
&=&\gamma^{-\alpha}M_n^{-\alpha}l\left(\gamma M_n(1+\vep_n)\right)\left\{1-(1+\vep_n)^{-\alpha}\right\}\\
&&\,\,\,\,\,\,+P(H>\gamma M_n)\left\{1-\frac{l\left(\gamma M_n(1+\vep_n)\right)}{l(\gamma M_n)}\right\}\,.
\end{eqnarray*}
The first term on the right hand side is clearly $O(\vep_nP(H>M_n))$, which by \eqref{A.1}, is $o\left(P(H>M_n)^{2-\beta}\right)$. By \eqref{A.3}, it follows that the second term is also $o\left(P(H>M_n)^{2-\beta}\right)$. This shows \eqref{l1.eq5}, and thus completes the proof of \eqref{l2.eq3}.

Next, we show that for all $\eta\in\bbr$, as $n\longrightarrow\infty$,
\begin{equation}\label{l2.eq1}
\sqrt{\tilde k_n}\log\frac{X_{(n,[\tilde k_n+\eta\tilde k_n^{1/2}])}}{X_{(n,\tilde k_n)}}\prob-\frac\eta\alpha\,.
\end{equation}
Let $(u_n)$ be a sequence of positive integers satisfying \eqref{l4.eq1}
 For $n$ large enough so that $1\le u_n<[n^{1-\beta}u_n^\beta]\le n$ and $1\le u_n<[n^{1-\beta}u_n^\beta]+\eta[n^{1-\beta}u_n^\beta]^{1/2}\le n$, the conditional distribution of $\left({X_{(\tilde k_n)}},{X_{([\tilde k_n+\eta\tilde k_n^{1/2}])}}\right)$ given that $U_n=u_n$ is same as the (unconditional) distribution of
$$
\left(Y_{(n-u_n,[n^{1-\beta}u_n^\beta]-u_n)},Y_{(n-u_n,[n^{1-\beta}u_n^\beta]+\eta[n^{1-\beta}u_n^\beta]^{1/2}-u_n)}\right)\,,
$$
where $\{Y_{(n,j)}:1\le j\le n\}$ is as defined in Lemma \ref{l4}, with $\tilde M_n$ as in \eqref{l2.eq4}.
Define $v_n$ as in \eqref{l2.eq2}
By Lemma \ref{l4}, it follows that
$$
\sqrt{v_n}\left(\frac1{v_n}\sum_{i=1}^n\delta_{Y_{n-u_n,i}/b((n-u_n)/v_n)}(y^{-1/\alpha},\infty]-y\right)\Longrightarrow W(y)
$$
in $D[0,\infty)$.
Using Vervaat's lemma, it follows that
\begin{equation}\label{l2.eq5}
\sqrt{v_n}\left[\left(\frac{Y_{(n-u_n,[v_nx])}}{b((n-u_n)/v_n)}\right)^{-\alpha}-x\right]\Longrightarrow-W(x)
\end{equation}
in $D[0,\infty)$. From here, we conclude that
$$
\left(\sqrt{v_n}\left[\left(\frac{Y_{(n-u_n,[v_ns_n])}}{b((n-u_n)/v_n)}\right)^{-\alpha}-s_n\right],{\sqrt v_n}\left[\left(\frac{Y_{(n-u_n,v_n)}}{b((n-u_n)/v_n)}\right)^{-\alpha}-1\right]\right)
$$
$$
\Longrightarrow(-W(1),-W(1))\,,
$$
where $s_n:=1+\eta v_n^{-1}[n^{1-\beta}u_n^\beta]^{1/2}$. Since the limit process is $C[0,\infty)\times C[0,\infty)$ valued, this can be done using Skorohod's Theorem (Theorem 2.2.2 in \cite{borkar:1995}). Using the Delta method with $x\mapsto-\frac1\alpha\log x$, it follows that
$$
\left(\sqrt{v_n}\left\{\log\frac{Y_{(n-u_n,[v_ns_n])}}{b((n-u_n)/v_n)}+\frac1\alpha\log s_n\right\},\sqrt{v_n}\log\frac{Y_{(n-u_n,v_n)}}{b((n-u_n)/v_n)}\right)
$$
$$
\Longrightarrow\left(\frac1\alpha W(1),\frac1\alpha W(1)\right)\,.
$$
Since,
$$
\lim_{n\to\infty}\sqrt{v_n}\log s_n=\eta\,,
$$
it follows that
$$
\sqrt{v_n}\log\frac{Y_{(n-u_n,[n^{1-\beta}u_n^\beta]-u_n)}}{Y_{(n-u_n,[n^{1-\beta}u_n^\beta]+\eta[n^{1-\beta}u_n^\beta]^{1/2}-u_n)}}\prob-\frac\eta\alpha\,.
$$
What we have shown is that whenever $(u_n)$ is a sequence satisfying \eqref{l4.eq1}, the conditional distribution of the left hand side of \eqref{l2.eq1} given $U_n=u_n$ converges weakly to $-\eta/\alpha$. By an appeal to Lemma \ref{l5}, this shows \eqref{l2.eq1}.

Coming to the proof of \eqref{l2.claim}, note that
\begin{eqnarray*}
&&\sqrt{\tilde k_n}\left[h(\hat k_n,n)-h(\tilde k_n,n)\right]\\
&=&\frac1{\sqrt{\tilde k_n}}\left[\sum_{i=1}^{\hat k_n}\log\frac{X_{(i)}}{X_{(\tilde k_n)}}-\sum_{i=1}^{\tilde k_n}\log\frac{X_{(i)}}{X_{(\tilde k_n)}}\right]+\frac{\hat k_n}{\sqrt{\tilde k_n}}\log\frac{X_{(\tilde k_n)}}{X_{(\hat k_n)}}\\
&&+\sqrt{\tilde k_n}\left(\frac1{\hat k_n}-\frac1{\tilde k_n}\right)\sum_{i=1}^{\hat k_n}\log\frac{X_{(i)}}{X_{(\hat k_n)}}\\
&=:&A+B+C\,.
\end{eqnarray*}
Clearly,
$$
C=\sqrt{\tilde k_n}\left(1-\frac{\hat k_n}{\tilde k_n}\right)h(\hat k_n,n)\prob0\,,
$$
the convergence in probability following from \eqref{l2.eq3} and the fact that
$$h(\hat k_n,n)\prob1/\alpha\,,$$
which has been shown in \cite{chakrabarty:samorodnitsky:2009}. For showing that $B\prob0$, fix $\epsilon>0$ and let $\eta:=\epsilon\alpha/6$. Note that
$$
P(|B|>\epsilon)
$$
$$
\le P\left[\frac{\hat k_n}{\tilde k_n}>2\right]+P\left[\sqrt{\tilde k_n}\left|\frac{\hat k_n}{\tilde k_n}-1\right|>\eta\right]+P\left[\sqrt{\tilde k_n}\log\frac{X_{(\tilde k_n-\eta\tilde k_n^{1/2})}}{X_{(\tilde k_n+\eta\tilde k_n^{1/2})}}>3\frac\eta\alpha\right]\,.
$$
By \eqref{l2.eq3} and \eqref{l2.eq1}, it follows that $B\prob0$. Since for $0<\epsilon<1$,
$$
P(|A|>\epsilon)\le P\left[\sqrt{\tilde k_n}\left|\frac{\hat k_n}{\tilde k_n}-1\right|>\epsilon\right]+P\left[\log\frac{X_{(\tilde k_n-\tilde k_n^{1/2})}}{X_{(\tilde k_n+\tilde k_n^{1/2})}}>1\right]\,,
$$
it is immediate that $A\prob0$. This completes the proof.
\end{proof}

\begin{proof}[Proof of Theorem \ref{t1}] In view of Lemma \ref{l2}, it suffices to show that
\begin{equation}\label{t1.eq1}
\sqrt{\tilde k_n}\left(h(\tilde k_n,n)-\frac1\alpha\right)\Longrightarrow N\left(0,\frac1{\alpha^2}\right)\,.
\end{equation}
Define
\begin{eqnarray*}
S_1&:=&\sum_{i=1}^{U_n}\log\frac{X_{(i)}}{X_{(\tilde k_n)}}\,\\
S_2&:=&\sum_{i=U_n+1}^{\tilde k_n}\log\frac{X_{(i)}}{X_{(\tilde k_n)}}\,
\end{eqnarray*}
and note that on the set $\{U_n\le\tilde k_n\}$,
$$
h(\tilde k_n,n)=\frac1{\tilde k_n}(S_1+S_2)\,.
$$
Let $u_n$ be a sequence of integers satisfying \eqref{l4.eq1} and define $v_n$ and $\tilde M_n$ as in \eqref{l2.eq2} and \eqref{l2.eq4}. For $n$ large enough, note that
$$
[S_2|U_n=u_n]\eid\sum_{i=1}^{v_n}\log\frac{Y_{(n-u_n,i)}}{Y_{(n-u_n,v_n)}}=:\tilde S_2\,,
$$
where $\{Y_{(n,j)}:1\le j\le n\}$ is as defined in the statement of Lemma \ref{l3}. By Lemma \ref{l3}, it follows that
$$
\sqrt{v_n}\left(\frac1{v_n}\tilde S_2-\frac1\alpha\right)\Longrightarrow N\left(0,\frac1{\alpha^2}\right)\,.
$$
This along with the fact that
$$
\sqrt{v_n}\tilde S_2\left(\frac1{[n^{1-\beta}u_n^\beta]}-\frac1{v_n}\right)=-\frac{\tilde S_2}{[n^{1-\beta}u_n^\beta]}\frac{u_n}{\sqrt{v_n}}=O_p(1)o(1)\,,
$$
shows that
$$
\left[\left.\sqrt{\tilde k_n}\left(\frac1{\tilde k_n}S_2-\frac1\alpha\right)\right|U_n=u_n\right]\Longrightarrow N\left(0,\frac1{\alpha^2}\right)\,.
$$
Since this is true for all sequence of integers $(u_n)$ satisfying \eqref{l4.eq1}, by Lemma \ref{l5} it follows that
$$
\sqrt{\tilde k_n}\left(\frac1{\tilde k_n}S_2-\frac1\alpha\right)\Longrightarrow N\left(0,\frac1{\alpha^2}\right)\,.
$$
On the set $\{1\le X_{(1)}\le2M_n\}$,
\begin{eqnarray*}
\frac{S_1}{\sqrt{\tilde k_n}}&\le&\frac{U_n\log(2M_n)}{\sqrt{\tilde k_n}}\\
&=&O_p\left(n^{1/2}P(H>M_n)^{1-\beta/2}\log M_n\right)\\
&=&o_p(1)\,.
\end{eqnarray*}
Since the probability of that set converges to one, it follows that
$$
\frac{S_1}{\sqrt{\tilde k_n}}\prob0\,.
$$
This completes the proof.
\end{proof}

\section{Second order regular variation}\label{sec:secondorder} In this section, we show that if the tail of $H$ is second order regularly varying, and $L$ is sufficiently light-tailed, then the hypotheses of Theorem \ref{t1} hold. By the tail being second order regularly varying, we mean that there is a function $A:(0,\infty)\longrightarrow(0,\infty)$ which is regularly varying with index $\rho\alpha$ where $\rho<0$, such that
\begin{equation}\label{eq:secondorder}
\lim_{t\to\infty}\frac{\frac{P(H>tx)}{P(H>t)}-x^{-\alpha}}{A(t)}=x^{-\alpha}\frac{x^{\rho\alpha}-1}{\rho/\alpha}
\end{equation}
for all $x>0$; see (2.3.24) in \cite{dehaan:ferreira:2006}.

\begin{theorem}\label{t2} Suppose that
\begin{equation}\label{t2.eq1}
\max(1-1/\alpha,0)<\beta<1\,,
\end{equation}
all moments of $L$ are finite, $M_n$ satisfies assumptions B and C, and the tail of $H$ is second order regularly varying so that the second order parameter $\rho$ satisfies
$$
\rho<-\frac{1-\beta}{\beta}\,.
$$
Then, \eqref{t1.claim} holds.
\end{theorem}

\begin{proof}In view of Theorem \ref{t1}, it suffices to check that assumptions A, D and E hold.
By Theorem 2.3.9 in \cite{dehaan:ferreira:2006}, it follows that given $\epsilon,\delta>0$, there exist $t_0>1$ such that whenever $t,tx\ge t_0$,
\begin{equation}\label{t2.eq2}
\left|\frac{\frac{P(H>tx)}{P(H>t)}-x^{-\alpha}}{A(t)}-x^{-\alpha}\frac{x^{\rho\alpha}-1}{\rho/\alpha}\right|\le\epsilon x^{-\alpha+\rho\alpha}\max(x^\delta,x^{-\delta})\,.
\end{equation}
Note that \eqref{t2.eq2} holds with a possibly different $A(t)$ from that in \eqref{eq:secondorder}. However, this $A$ is also regularly varying with index $\rho\alpha$. For the rest of the proof, by $A(\cdot)$, we shall mean the one for which \eqref{t2.eq2} holds.

We start with showing that
\begin{equation}\label{t2.eq3}
\sqrt{v_n}=o\left(A(b(n/v_n))^{-1}\right)\,,
\end{equation}
whenever $v_n$ is a sequence satisfying \eqref{D.eq1}.
Let
$$
\eta:=-\rho\beta-(1-\beta)\,.
$$
The upper bound on $\rho$ implies $\eta>0$. Note that $A(b(\cdot))$ varies regularly with index $\rho$ and $n/v_n\sim P(H>\gamma M_n)^{-\beta}$.
Thus, there is a slowly varying function $\bar l$ so that
\begin{eqnarray*}
A\left(b(n/v_n)\right)^{-1}&\sim&\bar l(M_n)P(H>M_n)^{\rho\beta}\\
&\gg& P(H>M_n)^{\eta+\rho\beta}\\
&=&\frac{n^{1/2}P(H>M_n)^{\beta/2}}{n^{1/2}P(H>M_n)^{1-\beta/2}}\\
&\gg&n^{1/2}P(H>M_n)^{\beta/2}\\
&\sim&\gamma^{\alpha\beta/2}\sqrt{v_n}\,,
\end{eqnarray*}
the inequality in the second last line following from Assumption C. This shows
\eqref{t2.eq3}.

Now, we show that assumptions D and E hold.
Let $$\vep_n:=A(b(n/v_n))\wedge(1/2)\,.$$ Clearly $1>\vep_n>0$ for all $n$. Recall from \eqref{bdefn} that $z<b(y)$ iff $P(H>z)^{-1}<y$. Thus,
$$
\frac1{P\left(H>(1-\vep_n)b(n/v_n)\right)}<\frac n{v_n}\le\frac1{P\left(H>b(n/v_n)\right)}\,.
$$
Let $\delta>0$ be such that $\rho\alpha+\delta<0$. Let $t_0$ be such that whenever $t,tx\ge t_0$, \eqref{t2.eq2} holds with $\epsilon=1$ and this $\delta$. Fix $0<T<\infty$. Let $N$ be such that for $n\ge N$, $b(n/v_n)>2t_0\vee t_0/T$. Thus, there is $C<\infty$, whose value may change from line to line, depending only on $T$, so that for $n\ge N$ and $x\ge T$,
$$
\left|\frac{P(H>b(n/v_n)x)}{P(H>b(n/v_n))}-x^{-\alpha}\right|\le CA(b(n/v_n))x^{-\alpha+\rho\alpha+\delta}\le CA(b(n/v_n))x^{-\alpha}
$$
the second inequality following since $\rho\alpha+\delta<0$, and similarly
\begin{eqnarray*}
&&\sup_{T\le x<\infty}\left|\frac{P(H>b(n/v_n)x)}{P(H>(1-\vep_n)b(n/v_n))}-x^{-\alpha}(1-\vep_n)^\alpha\right|\\ &\le&CA((1-\vep_n)b(n/v_n))\left(\frac x{1-\vep_n}\right)^{-\alpha}\\
&\le&CA(b(n/v_n))x^{-\alpha}\,.
\end{eqnarray*}
Since
$$
(1-\vep_n)^\alpha-1=O(\vep_n)=O(A(b(n/v_n)))\,,
$$
it follows that there is (a possibly different) $C<\infty$ so that
for all $x\ge T$,
$$
\left|\frac n{v_n}{P(H>b(n/v_n)x)}-x^{-\alpha}\right|\le CA(b(n/v_n))x^{-\alpha}\,.
$$
This in view of \eqref{t2.eq3} shows that assumptions D and E hold.

Finally, we show that Assumption A holds. By \eqref{t2.eq1}, it follows that
$$
1-\alpha(1-\beta)>0\,.
$$
Let $p>0$ be such that
$$
\frac{\alpha(1-\beta)}p<1-\alpha(1-\beta)\,.
$$
This choice of $p$ ensures that
\begin{equation}\label{t2.eq4}
\frac{\alpha(2-\beta)}p<1-\alpha\left(1-\beta-\frac1p\right)\,.
\end{equation}
Note that $xP(H>x)^{1-\beta-1/p}$ is regularly varying with index $1-\alpha(1-\beta-1/p)$ and $P(H>x)^{-(2-\beta)/p}$ is regularly varying with index $\alpha(2-\beta)/p$. Thus, by \eqref{t2.eq4} it follows that
\begin{eqnarray*}
M_nP(H>M_n)^{1-\beta-1/p}
\gg P(H>M_n)^{-(2-\beta)/p}
\gg n^{1/p}\,,
\end{eqnarray*}
the last inequality following from Assumption C. Thus
$$
n^{1/p}P(H>M_n)^{1/p}M_n^{-1}\ll P(H>M_n)^{1-\beta}\,.
$$
Let $(\vep_n)$ be such that
$$
n^{1/p}P(H>M_n)^{1/p}M_n^{-1}\ll\vep_n\ll P(H>M_n)^{1-\beta}\,.
$$
Clearly, \eqref{A.1} holds with this choice of $(\vep_n)$. For \eqref{A.2}, note that since $EL^p<\infty$,
$$
nP(H>M_n)P(L>\vep_nM_n)=O\left(nP(H>M_n)\vep_n^{-p}M_n^{-p}\right)=o(1)\,.
$$
This shows \eqref{A.2}. Finally, for \eqref{A.3}, choose $\delta>0$ so that $\rho\alpha+\delta<0$. Let $t_0$ be such that \eqref{t2.eq2} holds with this $\delta$ and $\epsilon=1$. Thus, as $n\longrightarrow\infty$,
\begin{eqnarray*}
\left|\frac{l(\gamma M_n(1+\vep_n))}{l(\gamma M_n)}-1\right|
&=&O\left(\left|\frac{P(H>\gamma M_n(1+\vep_n))}{P(H>\gamma M_n)}-(1+\vep_n)^{-\alpha}\right|\right)\\
&=&O\left(A(M_n)M_n^{-\alpha+\rho\alpha+\delta}\right)\\
&=&o\left(P(H>M_n)^{1-\beta}\right)\,,
\end{eqnarray*}
the last step following from the observations that
$$
P(H>M_n)^{-(1-\beta)}A(M_n)M_n^{-\alpha+\rho\alpha+\delta}
$$
$$
=\left\{M_n^{-\alpha}P(H>M_n)^{-(1-\beta)}\right\}M_n^{\rho\alpha+\delta}A(M_n)
$$
and that each of the three terms on the right hand side go to zero. This shows that Assumption A holds and thus completes the proof.
\end{proof}

\section{How to use this in practice}\label{sec:practice} While the assumptions A, D and E mentioned in Section \ref{sec:result} can be verified by assuming the second order regular variation and that all moments of $L$ are finite, one still needs a way to check assumptions B and C in practice. Statistical tests for checking Assumption B have been discussed in Chakrabarty and Samorodnitsky (2009). For checking Assumption C, which means that $M_n$ grows fast enough, one can use the facts that
$$
\frac{X_{(1)}}{M_n}\prob1\,,
$$
and
$$
\frac{\sum_{j=1}^n\one(X_j>\gamma M_n)}{nP(H>\gamma M_n)}\prob1
$$
for $0<\gamma<1$. These facts have been proved in Chakrabarty and Samorodnitsky (2009). An immediate consequence of these is that if Assumption C holds, then
$$
n\left(\frac1n\sum_{j=1}^n\one(X_j>\gamma X_{(1)})\right)^{2-\beta}(\log X_{(1)})^2\prob0\,.
$$
Thus, a natural thing to do is to choose $\beta$ (if possible) such that the above is satisfied.

We would like to mention at this point that from the point of view of using Theorem \ref{t2}, some issues remain unsorted. One of them is how does one ensure \eqref{t2.eq1}. A naive method would be to first get a ``rough'' estimate of $\alpha$ and then choose $\beta$ to satisfy the above. However, it is not clear at the moment that this is going to work. The other unsorted issue is that of checking the second order regular variation in the data and that $\rho<-(1-\beta)/\beta$. But then part of this is also a criticism for the Hill statistic applied to untruncated data; the same is known to be asymptotically normal only under some form of second order regular variation.

\section{Acknowledgement} The author is immensely grateful to his adviser Gennady Samorodnitsky for many helpful discussions.

\bibliography{D:/Mydata/work_arijit/res_ht/bibfile}
\bibliographystyle{D:/Mydata/work_arijit/res_ht/apalike}

\end{document}